\documentclass[11pt]{article}
\usepackage{graphicx}
\usepackage{amsfonts, amsmath, amssymb, amsthm, mathtools}
\usepackage{stmaryrd}
\usepackage{xcolor}
\usepackage{dsfont}
\usepackage{tikz-cd}
\usepackage{enumitem}

\usepackage{hyperref}
\hypersetup{
    colorlinks=true,
    linkcolor=blue,
    filecolor=magenta,      
    urlcolor=cyan,
}
\numberwithin{equation}{section}

\newcommand{\N}{\mathbb{N}}
\newcommand{\R}{\mathbb{R}}
\newcommand{\C}{\mathbb{C}}

\newcommand{\bord}{\partial\Omega}

\newcommand{\Hp}{\mathcal{B}(\bord)}
\newcommand{\Hm}{\mathcal{B}'(\bord)}

\newcommand{\Tr}{\operatorname{Tr}}
\newcommand{\tr}{\operatorname{tr}}

\newcommand{\ddn}[1]{\frac{\partial #1}{\partial\nu}}

\newcommand{\Bcal}{\mathcal{B}}

\newcommand{\Ucal}{\mathcal{U}}

\newcommand{\Tcal}{\mathcal{T}}
\newcommand{\dx}{\mathrm{d}x}
\newcommand{\dy}{\mathrm{d}y}

\newcommand{\diamO}{\mathrm{diam}(\Omega)}
\newcommand{\widO}{\mathrm{wid}(\Omega)}

\newtheorem{theorem}{Theorem}[section]
\newtheorem{lemma}[theorem]{Lemma}
\newtheorem{definition}[theorem]{Definition}
\newtheorem{proposition}[theorem]{Proposition}
\newtheorem{corollary}[theorem]{Corollary}
\newtheorem{remark}[theorem]{Remark}

\title{Poincaré-Steklov operator and Calder\'on's problem on extension domains}
\author{Gabriel Claret\thanks{Laboratoire MICS and Fédération de Mathématiques de CentraleSupélec, CentraleSupélec, Université Paris-Saclay, France. \texttt{gabriel.claret@centralesupelec.fr}} \and Michael Hinz\thanks{University of Bielefeld, Germany. \texttt{mhinz@math.uni-bielefeld.de}} \and Anna Rozanova-Pierrat\thanks{Laboratoire MICS and Fédération de Mathématiques de CentraleSupélec, CentraleSupélec, Université Paris-Saclay, France. \texttt{anna.rozanova-pierrat@centralesupelec.fr}}}
\date{\today}

\begin{document}

\maketitle

\begin{abstract}
We consider Calder\'on's problem on a class of Sobolev extension domains containing non-Lipschitz and fractal shapes.
We generalize the notion of Poincaré-Steklov (Dirichlet-to-Neumann) operator for the conductivity problem on such domains.
From there, we prove the stability of the direct problem for bounded conductivities continuous near the boundary. Then, we turn to the inverse problem and prove its stability at the boundary for Lipschitz conductivities, which we use to identify such conductivities on the domain from the knowledge of the Poincaré-Steklov operator. Finally, we prove the stability of the inverse problem on the domain for $W^{2,\infty}$ conductivities constant near the boundary. The last two results are valid in dimension $n\ge3$.
\end{abstract}

\tableofcontents

\section{Introduction}

We consider Calder\'on's problem on a class of $H^1$-Sobolev extension domains.
That class contains not only Lipschitz, but also non-Lipschitz shapes, such as domains with fractal or even multi-fractal boundaries, and boundaries of changing Hausdorff dimension.
To do so, we generalize the notion of Poincaré-Steklov (Dirichlet-to-Neumann) operator for the conductivity problem on such domains for positive, bounded conductivities continuous near the boundary.
We prove the stability of the direct problem for such conductivities on our class of extension domains.
Afterwards, we consider Calder\'on's inverse problem.
We prove the stability of the determination at the boundary of Lipschitz conductivities from the knowledge of the Poincaré-Steklov operator, which we use to identify such conductivities on the domain.
Finally, we prove the stability of the determination on the domain of $W^{2,\infty}$ conductivities constant near the boundary.
The last two results rely on a correspondance between the conductivity and the Schrödinger problems, and the existence of `high frequency' solutions to the latter, which holds in dimension $n\ge3$.

Whether it be in the medical field or in the study of metal alloys, the matter of identifying a conductivity from boundary measurements arises in a variety of applications.
Indeed, an inclusion (cancer tumour, impurity in a metal, buried object, and so on) can be seen as a change in the physical properties of a medium, namely its conductivity.
Like cancer tumours which can be highly vascularized, real life objects display irregularities at various scales~\cite{mandelbrot_how_1967,mandelbrot_fractal_1983}. For that matter, it seems relevent to model them using irregular shapes and fractals, which can capture the complexity of their geometries.
In the study of boundary value problems, the consideration of non-Lipschitz shapes demands to overcome several hindrances, among which the definition of the boundary values themselves.
The trace to the boundary can no longer be understood as an element of the usual trace Sobolev space $H^{\frac12}(\bord)$, and the normal derivative cannot be defined using the outward normal vector field, for there is no guarantee that field can be defined anywhere on the boundary.
The framework of Soblev extension domains allows to overcome those obstacles and define a trace operator~\cite{biegert_traces_2009} beyond the Lipschitz case which, in particular, does not rely on the existence of a specific boundary measure.
Unlike the previous generalizations of the notion of trace to non-Lipschitz boundaries -- to $d$-set boundaries in~\cite{wallin_trace_1991} and to $d$-upper regular boundaries in~\cite{hinz_non-lipschitz_2021} --, that definition of trace is built on the capacity with respect to $H^1(\R^n)$, which leads to that measure-free construction.
From there, the normal derivatives can be defined by duality using Green's formula~\cite{claret_layer_2024,lancia_transmission_2002}, which allows to adopt a variational approach and consider boundary value problems on extension domains.

The conductivity problem is described by the following equation:
\begin{equation}\label{Eq:Conductivity}
\nabla\cdot(\gamma\nabla u)=0\quad\mbox{on }\Omega,
\end{equation}
where $\gamma:\overline\Omega\to]0,+\infty[$. In an electrical analogy, $\gamma$ can be seen as the electric conductivity of the domain $\Omega$, and $u$ as the electric potential.
The problem one wishes to solve is to identify $\gamma$ based on the knowledge of the pairs formed by the potential at the boundary and the current across that boundary for all the solutions to~\eqref{Eq:Conductivity}.
This boundary data is expressed as the Poincaré-Steklov (or Dirichlet-to-Neumann) operator, formally defined in the following way:
\begin{equation*}
\Lambda^\gamma:\left. u\right|_{\bord}\longmapsto \left. \gamma\ddn{u}\right|_{\bord},
\end{equation*}
where $u$ is a solution to~\eqref{Eq:Conductivity} (see Definition~\ref{Def:P-S} for a proper definition of that operator).
The inverse problem consisting in recovering the conductivity $\gamma$ from the boundary values of the solutions to Eq.~\eqref{Eq:Conductivity} (in the form of $\Lambda^\gamma$) -- think of the electrical impedance tomography for instance~\cite{ammari_electrical_2008, henderson_impedance_1978} -- is often referred to as \textit{Calder\'on's problem}, for it was raised by A. P. Calder\'on in 1980~\cite{calderon_inverse_1980}.
Since then, it has been studied under various assumptions regarding the regularity of the domain and of the conductivity (see for instance~\cite{alessandrini_singular_1990,caro_global_2016,kohn_determining_1984}). 

The results on Calder\'on's problem can be split into several categories, based on two considerations.
On the one hand, what is given and what is unknown: we will refer to the direct problem as the determination of the boundary data $\Lambda^\gamma$ from the knowledge of $\gamma$, while the inverse problem consists in recovering $\gamma$ from $\Lambda^\gamma$.
One the other hand, the level of accuracy: identification results conclude to the uniqueness of the unknown given the data of the problem, while stability results allow to quantify the impact a perturbation of the data has on the determination of the unknown.
Without making assumptions on the conductivity $\gamma$, Calder\'on's inverse problem is ill-posed, as several conductivities can be associated to the same Poincaré-Steklov operator; a counter-example can be found in~\cite[p. 156]{alessandrini_stable_1988}.
Therefore, all inverse determination and stability results are conditional, in the sense that the conductivity is assumed to have a certain regularity, and to be bounded below (and above) by given positive constants, as in~\cite{alessandrini_stable_1988,sylvester_global_1987} for instance. 

The purpose of this work is to generalize identification and stability results for Calder\'on's direct and inverse problems on a class of extension domains which we call \emph{admissible} (see Definition~\ref{def:Admissible}).
The first identification result can be found in~\cite{kohn_determining_1984}, where Kohn and Vogelius proved that a real-analytic conductivity is uniquely determined by the associated Poincaré-Steklov operator, assuming the domain is $C^\infty$.
Shortly after, that result was generalized to piecewise analytic conductivities~\cite{kohn_determining_1985}.
The determination of a $C^\infty$ conductivity on a $C^\infty$ domain was proved by Sylvester and Uhlmann in~\cite{sylvester_global_1987} in dimension $n\ge 3$, transforming Eq.~\eqref{Eq:Conductivity} into an equivalent Schrödinger equation (see Eq.~\eqref{Eq:Schrodinger}) and using so-called complex geometrical optics (CGO) solutions -- or high frequency solutions -- to the latter.
A similar transformation was used in~\cite{caro_global_2016} and allowed to identify a Lipschitz conductivity on a Lipschitz domain; we adapt their method to identify a Lipschitz conductivity on an admissible domain satisfying the exterior corkscrew condition~\eqref{eq:Corkscrew} in Theorem~\ref{Th:Inverse-Identification-Domain}.
The extra condition allows to prove boundary stability estimates, on which the approach relies.
Moreover, stability estimates for the direct problem (the determination of $\Lambda^\gamma$ from $\gamma$) are proved in~\cite{sylvester_global_1987} for bounded conductivities and for Lipschitz conductivities~\cite{sylvester_inverse_1988}.
We generalize those estimates to admissible domains for bounded conductivities continuous near the boundary in Theorem~\ref{Th:Calderon-Direct}, relying notably on norm estimates for traces modulated by Lipschitz functions (Lemma~\ref{Lem:InequalityB}).
Boundary stability estimates for the inverse problem were also obtained in~\cite{sylvester_global_1987} in the case of $C^\infty$ conductivities, using the pseudo-differential properties of the Poincaré-Steklov operator on smooth domains. 
The conductivity-Schrödinger equivalence and the CGO solutions allowed to make headway in proving identification and stability results for Calder\'on's inverse problem in dimension $n\ge3$:
their use allowed to prove stability estimates on smooth domains for $H^{s+2}$ conductivities, $s>\frac n2$, with a logarithmic modulus of continuity~\cite{alessandrini_stable_1988}, and later on to yield the same kind of estimates at the boundary~\cite{alessandrini_singular_1990} and on the domain~\cite{alessandrini_determining_1991} for $W^{2,\infty}$ conductivities, this time on Lipschitz domains.
There, the loss of regularity of the domain is accounted for by considering a $C^\infty$ quasi-normal vector field (in the sense that the inner product with the normal field is bounded below by a positive constant), which allows to mimic a $C^\infty$ boundary.
One cannot define such a field in the case of a general admissible domain (for the normal vector field itself is not defined \textit{a priori}).
Instead, we assume our admissible domain satisfies the exterior corkscrew condition~\eqref{eq:Corkscrew} and consider a sequence of singular solutions with singularities approaching the boundary at a `suitable' rate.
This allows to yield boundary stability estimates on such admissible domains for Lipschitz conductivities in Theorem~\ref{Th:Calderon-Inv-Stability-Boundary}.
Finally, we use the CGO solutions once again to generalize the results from~\cite{alessandrini_singular_1990,alessandrini_determining_1991} and prove the stability of the determination of the conductivity on the domain in the case of an admissible domain and $W^{2,\infty}$ conductivities constant near the boundary in dimension $n\ge3$ in Theorem~\ref{Th:Inverse-Stability-Domain}. The method of the CGO solutions does not work for $n=2$, and in this article we don't treat the two-dimensional case. 

However, let us mention the kown results in the case $n=2$. In the two-dimensional case the identification results for $W^{2,p}$ conductivities, $p>1$, on Lipschitz domains were proved in~\cite{nachman_global_1996}, for $W^{1,p}$ conductivities, $p>2$, in~\cite{brown_uniqueness_1997}, and generalized to bounded conductivities in~\cite{astala_calderons_2006}.
The boundary regularity is not specified in~\cite{astala_calderons_2006}, however the use of the trace space $H^{\frac 12}(\bord)$ hints at a Lipschitz regularity of the domain.
Still in the plane, logarithmic stability estimates were proved in~\cite{liu_stability_1997} on $C^{1,1}$ domains for $W^{2,p}$ conductivities, $p\in]1,2[$, in~\cite{barcelo_stability_2001} on Lipschitz domains for $C^{1,\alpha}$ conductivities, $\alpha>0$, and in~\cite{barcelo_stability_2007} for Hölder continuous conductivities.

We also mention works of a different nature regarding Calder\'on's problem in which the boundary data is only known on part of the bounday: an identification result of $C^{2,\alpha}$ conductivities near the boundary of a smooth plane domain with partial data is proved in~\cite{imanuvilov_calderon_2010}.
In~\cite{ouhabaz_milder_2018}, the partial Poincaré-Steklov operator is used to determine a symmetric elliptic operator on a Lipschitz domain up to unitary equivalence of the coefficients.

Throughout this paper, $n\ge2$ is the dimension of the ambient space. The $n$-dimensional Lebesgue measure of a Borel set $U\subset\R^n$ will be denoted by $|U|$.
Its complement will be denoted by $U^c:=\R^n\backslash U$ and its closure by $\overline{U}$.
Its diameter will be denoted by $\mathrm{diam}(U):=\sup\{|x-y|\;|\;x,y\in U\}$.
The smallest distance between two parallel hyperplanes between which $U$ lies will be referred to as its width. Formally, it is defined as
\begin{multline*}
\mathrm{wid}(U):=\inf\{a>0\;|\;\exists\, \mathcal{H}_1,\mathcal{H}_2\mbox{ hyperplanes},\\
\mathrm{dist}(\mathcal{H}_1,\mathcal{H}_2)=a\quad\mbox{and}\quad U\subset \mathrm{conv}(\mathcal{H}_1,\mathcal{H}_2)\},
\end{multline*}
where `$\mathrm{dist}$' denotes euclidean distance and `$\mathrm{conv}$' denotes the convex hull.
For $x\in\R^n$ and $r>0$, $B_r(x)$ denotes the open ball of centre $x$ and radius $r$.
The term `domain' refers to a non-empty open connected set.
If $\Omega$ is a bounded domain, the set of all Lipschitz functions on $\Omega$ (and up to the boundary) will be denoted by $\mathrm{Lip}(\Omega)$.
The notation $c$ refers to a positive constant, which may change from one line to another.
If the constant $c$ depends on quantities $\alpha$, $\beta$, $\gamma$, ..., we will denote $c=c(\alpha,\beta,\gamma,...)$.

The rest of the paper is organised as follows: in Section~\ref{Sec:Framework}, we define the class of admissible domains and study the Poincaré-Steklov (Dirichlet-to-Neumann) operator for the conductivity problem in that framework.
In Section~\ref{Sec:Calderon-Problem}, we consider Calder\'on's problem on admissible domains: we focus on the stability of the direct problem in Subsection~\ref{Subsec:Direct-Problem}, and on the inverse problem in Subsection~\ref{Subsec:Inverse-Problem} (stability at the boundary, stability and identification on the domain).
Part of the proofs for the direct problem are similar to the classical case (notably~\cite{sylvester_inverse_1988}) and can be found in Appendix~\ref{A-Sec:Proofs}.
In Appendix~\ref{A-Sec:WP-Estimates}, we prove the well-posedness of the Dirichlet Laplace and conductivity problems on admissible domains along with the associated \textit{a priori} estimates, where we specify the dependencies of the constants involved.

\section{Functional framework}\label{Sec:Framework}

We begin with defining the functional framework in which this study is carried out, namely the classes of domains and conductivities, and the associated Poincaré-Steklov operators.

\subsection{Admissible domains}\label{Subsec:Admissible-Domains}

We define the class of domains considered, which we refer to as the class of admissible domains.

\begin{definition}[Admissible domains]\label{def:Admissible}
A bounded domain $\Omega$ of $\R^n$ is said to be an admissible domain if
\begin{enumerate}
\item[(i)] there exists a linear bounded extension operator $\mathrm{E}_\Omega:H^1(\Omega)\to H^1(\R^n)$: 
for all $u\in H^1(\Omega)$, $\mathrm{E}_\Omega u|_\Omega=u$ and $$\|\mathrm{E}_\Omega u\|_{H^1(\R^n)}\le c(n,\Omega)\|u\|_{H^1(\Omega)};$$
\item[(ii)] $\bord$ has positive capacity.
\end{enumerate}
\end{definition}

Domains satisfying condition (i) are called $H^1$-extension domains~\cite{hajlasz_sobolev_2008, jones_quasiconformal_1981}.
The notion of capacity used in (ii) refers to the capacity with respect to $H^1(\R^n)$, see~\cite[Section 2.1]{fukushima_dirichlet_2010},~\cite[Section 7.2]{mazya_boundary_1991} and~\cite[Section 2]{biegert_traces_2009}; the notions `quasi-everywhere' (q.e.) and `quasi-continuous' are to be understood accordingly.
All Lipschitz domains are $H^1$-extension domains; more generally, all $(\varepsilon,\delta)$-domains are $H^1$-extension domains~\cite[Theorem 1]{jones_quasiconformal_1981}. If $\Omega$ is a $(\varepsilon,\infty)$-domain and $\overline{\Omega}^c$ is non-empty, then $\Omega$ is admissible.
By~\cite[Theorem 2]{hajlasz_sobolev_2008}, any $H^1$-extension domain $\Omega$ is an $n$-set:
\begin{equation}\label{Eq:n-set}
\exists c>0,\;\forall x\in\Omega,\;\forall r\in]0,1],\quad |\Omega\cap B_r(x)|\ge c\,r^n.
\end{equation}
It follows that a domain with an outward cusp cannot satisfy the $H^1$-extension property, see Figure~\ref{Fig:AdmissibleDomains}.

\begin{figure}[t]
\centering
\includegraphics[height = 5.7 cm, width = 6 cm]{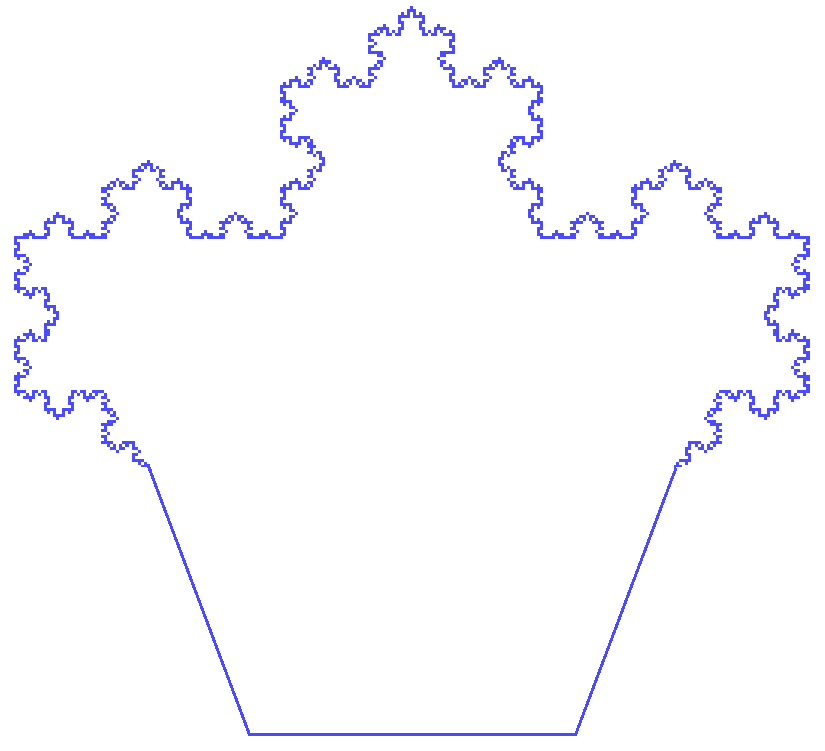}
\hfill
\includegraphics[height = 5.7 cm, width = 5.5 cm]{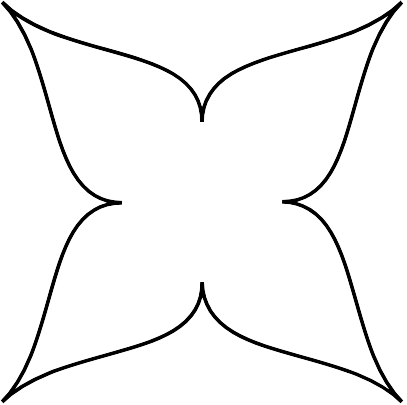}
\caption{On the left, an admissible domain of $\R^2$. The top part of its boundary consists of a Von Koch curve of Hausdorff dimension $\frac{\ln 4}{\ln 3}$, while the bottom part is of Hausdorff dimension $1$. On the right, a domain of $\R^2$ which is not an extension domain for it has outward cusps.}
\label{Fig:AdmissibleDomains}
\end{figure}

For $u\in H^1(\R^n)$, denote by $\tilde u$ a quasi-continuous representative of $u$. If $\Omega$ is an admissible domain, then $[\tilde{u}|_{\bord}]_{\Hp}$, which is the q.e. equivalence class of the pointwise restriction $\tilde{u}|_{\bord}$, depends on $u$ rather than on the choice of $\tilde u$. We denote by $\mathcal{B}(\partial\Omega)$ the space of all $[\tilde{u}|_{\bord}]_{\Hp}$ for $u\in H^1(\R^n)$. We define the trace operator on $\bord$ as in~\cite[Proposition 2.1]{claret_layer_2024}, following~\cite[Theorem 6.1, Remark 6.2 and Corollary 6.3]{biegert_traces_2009}.

\begin{definition}[Trace operator]\label{Def:Trace}
Let $\Omega$ be an admissible domain of $\R^n$ and let $\mathrm{E}_\Omega:H^1(\Omega)\to H^1(\R^n)$ be a linear bounded extension operator. The trace operator on $\bord$ is defined as
\begin{align*}
\Tr:H^1(\Omega)&\longrightarrow\Hp\\
u&\longmapsto [(\mathrm{E}_\Omega u)^\sim|_{\bord}]_{\Hp},
\end{align*}
where $(\mathrm{E}_\Omega u)^\sim$ is a quasi-continuous representative of $\mathrm{E}_\Omega u\in H^1(\R^n)$.
\end{definition}

If $u\in H^1(\Omega)$, a representative of $\Tr u$ is given for q.e. $x\in\bord$ by the following formula:
\begin{equation}\label{Eq:Tr-Representative}
\Tr u(x)=\lim_{r\to 0^+}\frac{1}{|B_r(x)\cap\Omega|}\int_{B_r(x)\cap\Omega}u\,\dx.
\end{equation}
The following theorem, stated in a similar form in~\cite{claret_layer_2024}, synthetizes results regarding the trace operator used throughout this paper.

\begin{theorem}[Trace theorem]\label{Th:Trace}
Let $\Omega$ be an admissible domain of $\R^n$. Then, the following assertions hold.
\begin{enumerate}[label=(\roman*)]
\item\label{Pt:Trace-norm} Endowed with the norm
\begin{equation}\label{normTr}
\left\|f\right\|_{\Hp}:=\min\{ \left\|v\right\|_{H^1(\Omega)}|\ f=\Tr v\},
\end{equation}
the space $\Hp=\Tr(H^1(\Omega))$ is a Hilbert space.
\item\label{Pt:Trace-kernel} $\operatorname{Ker}(\Tr)=H^1_0(\Omega):=\overline{C^\infty_0(\Omega)}^{\|\cdot\|_{H^1(\Omega)}}$ is the closure in $H^1(\Omega)$ of the set of indefinitely differentiable functions with compact support in $\Omega$.
If $\Omega$ is bounded, then $\|\cdot\|_{H^1_0(\Omega)}:=\|\nabla\cdot\|_{L^2(\Omega)^n}$ is a norm on $H^1_0(\Omega)$ and it holds
\begin{equation}\label{Eq:Poincaré}
\forall u\in H^1_0(\Omega),\quad \|u\|_{H^1_0(\Omega)}\le\|u\|_{H^1(\Omega)}\le c(\widO)\|u\|_{H^1_0(\Omega)}.
\end{equation}
\item\label{Pt:Trace-contraction} The trace operator $\Tr: H^1(\Omega) \to \Hp$ is a partial isometry with operator norm equal to $1$.
\item\label{Pt:Trace-isometry} $\tr:=\Tr|_{V_1(\Omega)}: V_1(\Omega) \to \Hp$ defines an isometry, where 
\begin{equation}\label{V1}
V_1(\Omega)=\big\{v\in H^1(\Omega)\;\big|\; (-\Delta+1)v=0 \mbox{ on }\Omega\big\}
\end{equation}
is the space of $1$-harmonic functions on $\Omega$, endowed with the standard $H^1$ norm. It holds:
\begin{equation}\label{EqV1andOthers}
V_1(\Omega)=\operatorname{Ker}(\Tr)^\perp =(H^1_0(\Omega))^\perp.
\end{equation}
\end{enumerate}
\end{theorem}

\begin{proof}
We refer to~\cite{claret_layer_2024} and the references therein, in particular to~\cite[Theorem~1]{hinz_existence_2021},\cite[Theorem~5.1]{hinz_non-lipschitz_2021},~\cite[Theorem~2]{dekkers_mixed_2022} for Point~\ref{Pt:Trace-norm} and~\cite[Corollary 2.3.1 and Example 2.3.1]{fukushima_dirichlet_2010} for Point~\ref{Pt:Trace-kernel}.
The dependencies of the Poincaré constant from~\eqref{Eq:Poincaré} can be found in~\cite[Subsection 6.26]{adams_sobolev_1975}.
Points~\ref{Pt:Trace-contraction} and~\ref{Pt:Trace-isometry} follow from the previous points and the use of Stampacchia's theorem~\cite[Theorem V.6]{brezis_analyse_1987}. 
\end{proof}

The class of admissible domains contains non-smooth domains (see Figure~\ref{Fig:AdmissibleDomains}) and there is no guarantee that the normal vector field can be defined anywhere on the boundary of a such domain.
For that matter, we turn to the same weak notion of normal derivative as in~\cite{claret_layer_2024}, defined as an element of the dual space $\Hm:=\mathcal{L}(\Hp,\R)$.

\begin{definition}[Weak normal derivative]\label{def:ddn}
Let $\Omega$ be an admissible domain of $\R^n$. Let $u\in H^1_\Delta(\Omega)$, where that space is defined as
\begin{equation}\label{def:H1D}
H^1_{\Delta}(\Omega):=\left\{v\in H^1(\Omega)\;|\;\Delta v\in L^2(\Omega)\right\}.
\end{equation}
The weak normal derivative of $u$ is the element $\psi\in\Hm$ such that for all $v\in H^1(\Omega)$, it holds:
\begin{equation}\label{EqGreenInt}
\langle\psi,\Tr v\rangle_{\Hm,\,\Hp}= \int_\Omega{(\Delta u)v\,\dx}+\int_\Omega{\nabla u\cdot\nabla v\,\dx},
\end{equation}
and we denote $\ddn{u}|_{\bord}:=\psi$ (or simply $\ddn u$ when there is no ambiguity).
\end{definition}

We refer to~\cite[Subsection 2.5]{claret_layer_2024} for properties of the normal derivation operator, including the proof of the following result on isometric properties of that operator~\cite[Corollary 2.17]{claret_layer_2024}.

\begin{proposition}\label{Prop:ddn-isom}
If $\Omega$ is an admissible domain of $\R^n$, then $\ddn{}:V_1(\Omega)\to\Hm$ is an isometry and onto when $\Hm$ is endowed with the subordinate norm to $\lVert\cdot\rVert_{\Hp}$, denoted by $\lVert\cdot\rVert_{\Hm}$.
\end{proposition}

\subsection{Poincaré-Steklov operator for conductivity problem}\label{Subsec:PS-Operator}


If $\Omega$ is an admissible domain, we define the following space of conductivities:
\begin{multline*}
\Tcal(\Omega):=\big\{\eta\in L^\infty(\Omega)\;\big|\;\operatorname{ess\,inf}\eta>0\quad\mbox{and}\\
\exists\, V \mbox{ neighborhood of }\bord\mbox{ in }\Omega,\; \eta|_{V}\in C^0(V)\big\}.
\end{multline*}
For simplicity, if $\eta\in\Tcal(\Omega)$, its essential infimum will be denoted by $\inf\eta$. 
The continuity condition from $\Tcal(\Omega)$ allows to define the quantity $\gamma\ddn u|_{\bord}$ unambiguously when $u\in H^1_\Delta(\Omega)$, so that the Poincaré-Steklov operator $\Lambda^\gamma$ will be well-defined in Definition~\ref{Def:P-S} below. An essential part of the definition of $\Lambda^\gamma$ is the well-posedness of the Dirichlet conductivity problem.

\begin{lemma}\label{Lem:DirichletConductivity-WP}
Let $\Omega$ be an admissible domain of $\R^n$. Let $\gamma\in \Tcal(\Omega)$. Then, for all $f\in\Hp$, the Dirichlet problem
\begin{equation}\label{Eq:ConductivityDir}
\begin{cases}
\nabla\cdot(\gamma\nabla u)=0 &\mbox{on }\Omega,\\
\Tr  u=f,
\end{cases}
\end{equation}
has a unique weak solution $u\in H^1(\Omega)$. That solution satisfies
\begin{equation}\label{Eq:DirichletEstimate}
\|u\|_{H^1(\Omega)}\le c(\widO)\left(1+\frac{\|\gamma\|_{L^\infty(\Omega)}}{\inf\gamma}\right)\|f\|_{\Hp}.
\end{equation}
\end{lemma}

\begin{proof}
The well-posedness of~\eqref{Eq:ConductivityDir} is proved in Appendix~\ref{A-Subsec:Conductivity}. Here, the estimate is refined using Eq.~\eqref{Eq:Conductivity-Lifted-Estimate}.
\end{proof}

It follows that the Poincaré-Steklov operator for Problem~\eqref{Eq:Conductivity} is well-defined as an operator from the trace space $\Hp$ to its dual $\Hm$.

\begin{definition}[Poincaré-Steklov operator]\label{Def:P-S}
Let $\Omega$ be an admissible domain of $\R^n$ and let $\gamma\in\Tcal(\Omega)$. The Poincaré-Steklov operator associated to Problem~\eqref{Eq:Conductivity} is defined as
\begin{align*}
\Lambda^\gamma:\Hp &\longrightarrow\Hm\\
\Tr u &\longmapsto \gamma\ddn u,
\end{align*}
where $u\in H^1(\Omega)$ satisfies $\nabla\cdot(\gamma\nabla u)=0$ weakly on $\Omega$.
\end{definition}

We generalize to the case of admissible domains properties of the Poincaré-Steklov operator which are well-known in the case of Lipschitz domains, see for instance~\cite[p. 168]{alessandrini_stable_1988} and~\cite[p. 253]{alessandrini_singular_1990}.
Results similar to Point (i) of the proposition below can also be found in~\cite{arendt_spectral_2007} in the Lipschitz case, in~\cite{arfi_dirichlet--neumann_2019} for $d$-set boundaries, in~\cite{ROZANOVA-PIERRAT-2021} for extension domains with Jonsson boundary measures, and in~\cite{claret_layer_2024} in the case of our admissible domains with $\gamma=1$.

\begin{proposition}\label{Prop:P-S}
Let $\Omega$ be an admissible domain of $\R^n$. Then,
\begin{enumerate}
\item[(i)] if $\gamma\in\Tcal(\Omega)$, then the Poincaré-Steklov operator $\Lambda^\gamma$ is linear, bounded and coincides with its adjoint; 

\item[(ii)] if $\gamma_1,\gamma_2\in \Tcal(\Omega)$ and $u_1,u_2\in H^1(\Omega)$ satisfy $\nabla\cdot(\gamma_i\nabla u_i)=0$ weakly on $\Omega$, $i\in\{1,2\}$, then it holds
\begin{equation*}
\langle(\Lambda^{\gamma_1}-\Lambda^{\gamma_2})\Tr u_1,\Tr u_2\rangle_{\Hm,\,\Hp}=\int_\Omega{(\gamma_1-\gamma_2)\nabla u_1\cdot\nabla u_2\,\dx}.
\end{equation*}
\end{enumerate}
\end{proposition}

\begin{proof}
Let us prove Point (i). The linearity of $\Lambda^\gamma$ follows from the definition, while the boundedness follows from~\eqref{Eq:DirichletEstimate}. Let $f,g\in\Hp$ and let $u_f,u_g\in H^1(\Omega)$ be the weak solutions to~\eqref{Eq:ConductivityDir} associated to $f$ and $g$ respectively. Then,
\begin{equation*}
\langle\Lambda^\gamma f,g\rangle_{\Hm,\,\Hp}=\left\langle\gamma\ddn{u_f},\Tr u_g\right\rangle_{\Hm,\,\Hp}=\int_\Omega{\gamma\nabla u_f\cdot\nabla u_g\,\dx},
\end{equation*}
which is symmetric in $(f,g)$. Since $\Hp$ is a Hilbert space, it can be identified with its bidual so that $(\Lambda^\gamma)^*:\Hp\to\Hm$ and $(\Lambda^\gamma)^*=\Lambda^\gamma$.

Let us prove Point (ii). By Green's formula, it holds for $(i,j)=(1,2)$ or $(2,1)$,
\begin{align*}
\langle \Lambda^{\gamma_i} \Tr  u_i,\Tr  u_j\rangle_{\Hm,\,\Hp} &=\left\langle\gamma_i\ddn{u_i},\Tr  u_j\right\rangle_{\Hm,\,\Hp}\\
&= \int_\Omega{\gamma_i\nabla u_i\cdot\nabla u_j\,\dx},
\end{align*}
and the result follows from Point (i).
\end{proof}

The Poincaré-Steklov operators will always be understood as elements of $\mathcal{L}(\Hp,\Hm)$.
Without further precisions, any operator norm of such operators will be with respect to that space.

\section{Calder\'on's problem on extension domains}\label{Sec:Calderon-Problem}

Having generalized the notion of Poincaré-Steklov operator in Subsection~\ref{Subsec:PS-Operator}, we consider Calder\'on's problem on admissible domains, which regards the link between the conductivity and the Poincaré-Steklov operator.

\subsection{Direct problem}\label{Subsec:Direct-Problem}

We begin with the direct problem and prove the continuity of the mapping associated to the direct conductivity problem, that is
\begin{equation*}
\gamma\in \Tcal(\Omega)\longmapsto \Lambda^\gamma\in\mathcal{L}(\Hp,\Hm),
\end{equation*}
where $\Omega$ is only assumed to be an admissible domain.
The approach used here is similar to~\cite[Section 3]{sylvester_inverse_1988}, where the stability of the direct problem is proved in the case of conductivities in $L^\infty(\Omega)$ (or $W^{1,\infty}(\Omega)$ at times) on a smooth domain.
As a matter of fact, some proofs in the case of an admissible domain can be derived directly from their smooth counterparts found there, up to adapting them to the formalism of extension domains.
Those proofs can be found in Appendix~\ref{A-Sec:Proofs}.
The main differences between our case and~\cite{sylvester_inverse_1988} come from the norm estimates of traces and normal derivatives modulated by Lipschitz functions (Lemmas~\ref{Lem:InequalityB} and~\ref{Lem:InequalityB'}), and in the estimates of the norms of normal derivatives in the proof of the main result of this part (Theorem~\ref{Th:Calderon-Direct}).
Another technical difference in the case of admissible domains is that $\mathrm{Lip}(\Omega)$, the space of Lipschitz function on $\Omega$ (and up to the boundary), may be strictly included in $W^{1,\infty}(\Omega)$~\cite[Theorem 7]{hajlasz_sobolev_2008}, and must therefore be distinguished.

Problem~\eqref{Eq:ConductivityDir} is an elliptic problem with non-homogeneous Dirichlet condition.
To study such a problem, it is usual to consider a lifted problem set on $H^1_0(\Omega)$, in which the non-homogeneous Dirichlet condition is replaced with a source term.
The following lemma allows to estimate the solution of the lifted problem by means of the source term and the conductivity.

\begin{lemma}\label{Lem:EstimateG}
Let $\Omega$ be an admissible domain of $\R^n$. Let $\gamma,\eta^+,\eta^-\in\Tcal(\Omega)$. Denote $\eta:=\eta^+-\eta^-$. Let $\varphi\in H^1(\Omega)$. Then the conductivity problem
\begin{equation}\label{Eq:CondProbG}
\begin{cases}
\nabla\cdot(\gamma\nabla u)=\nabla\cdot(\eta\nabla \varphi) &\mbox{on }\Omega,\\
\Tr  u=0,
\end{cases}
\end{equation}
has a unique weak solution $u\in H^1_0(\Omega)$, and it holds
\begin{equation}\label{Eq:EstimGrad}
\|u\|_{H^1_0(\Omega)}\le \frac{\|\eta\|_{L^\infty}}{\inf \gamma}\|\nabla \varphi\|_{L^2(\Omega)^n}.
\end{equation}
In addition, if $\gamma,\eta\in W^{1,\infty}(\Omega)$ and $\Delta \varphi\in L^2(\Omega)$, then
\begin{equation}\label{Eq:EstimLap}
\|\Delta u\|_{L^2(\Omega)}\le \frac{\|\eta\|_{L^\infty}}{\inf \gamma}\|\Delta \varphi\|_{L^2(\Omega)}+\left(\frac{\|\nabla\eta\|_{L^\infty}}{\inf \gamma}+\frac{\|\nabla\gamma\|_{L^\infty}\|\eta\|_{L^\infty}}{(\inf \gamma)^2}\right)\|\nabla \varphi\|_{L^2(\Omega)}.
\end{equation}
\end{lemma}


In our case, the solution is lifted by means of a harmonic function.

\begin{proposition}\label{Prop:EstimateV}
Let $\Omega$ be an admissible domain of $\R^n$. Let $\gamma\in\Tcal(\Omega)$. Let $f\in\Hp$. Denote by $u_\Delta$ and $u$ the unique weak solutions in $H^1(\Omega)$ to the Dirichlet conductivity problem~\eqref{Eq:ConductivityDir} associated to $1$ and $\gamma$ respectively. Then, for $v:=u-u_\Delta$, it holds
\begin{equation}\label{Eq:EstimGrad-v}
\|v\|_{H^1_0(\Omega)}\le c(\widO)\sqrt{\frac{\|\gamma\|_{L^\infty}}{\inf \gamma}}\|f\|_{\Hp},
\end{equation}
and
\begin{equation}\label{Eq:EstimGrad-u}
\|\nabla u\|_{L^2(\Omega)^n}\le
c(\widO)\left(1+\sqrt{\frac{\|\gamma\|_{L^\infty}}{\inf \gamma}}\right)\|f\|_{\Hp}.
\end{equation}
If in addition $\gamma\in W^{1,\infty}(\Omega)$, then
\begin{equation}\label{Eq:EstimLap-v}
\|\Delta v\|_{L^2(\Omega)}\le c(\widO)\frac{\|\nabla\gamma\|_{L^\infty}}{\inf \gamma}
\left(1+\sqrt{\frac{\|\gamma\|_{L^\infty}}{\inf \gamma}}\right)\|f\|_{\Hp}.
\end{equation}
\end{proposition}

This allows to quantify the impact a variation of the conductivities has on the lifted solution.

\begin{corollary}\label{Cor:EstimateDiff}
Let $\Omega$ be an admissible domain of $\R^n$. Let $\gamma_1,\gamma_2\in\Tcal(\Omega)$. Let $f\in\Hp$. Denote by $u_\Delta$, $u_1$ and $u_2$ the unique weak solutions in $H^1(\Omega)$ to the Dirichlet problem~\eqref{Eq:ConductivityDir} associated to $1$, $\gamma_1$ and $\gamma_2$ respectively. Then, denoting $v_i:=u_i-u_\Delta$ for $i\in\{1,2\}$, it holds
\begin{equation}\label{Eq:Estim-vDiff}
\|v_1-v_2\|_{H^1_0(\Omega)}\le c(\widO)\frac{\|\gamma_1-\gamma_2\|_{L^\infty}}{\inf \gamma_1}\left(1+\sqrt{\frac{\|\gamma_2\|_{L^\infty}}{\inf \gamma_2}}\right)\|f\|_{\Hp}.
\end{equation}
If in addition $\gamma_1,\gamma_2\in W^{1,\infty}(\Omega)$, then
\begin{multline}\label{Eq:EstimLapDiff}
\|\Delta (v_1-v_2)\|_{L^2(\Omega)}\le
c(\widO)\left(1+\sqrt{\frac{\|\gamma_2\|_{L^\infty}}{\inf \gamma_2}}\right)
\Bigg[\frac{\|\gamma_1-\gamma_2\|_{L^\infty}}{\inf \gamma_1}\times\\
\Bigg(\frac{\|\nabla\gamma_1\|_{L^\infty}}{\inf \gamma_1}
+\frac{\|\nabla\gamma_2\|_{L^\infty}}{\inf \gamma_2}\Bigg)
+\frac{\|\nabla(\gamma_1-\gamma_2)\|_{L^\infty}}{\inf \gamma_1}\Bigg]\|f\|_{\Hp}.
\end{multline}
\end{corollary}

The following lemma allows to estimate the norm of a trace modulated by a Lipschitz function.
In the case of a Lipschitz domain, the trace space $H^{\frac 12}(\bord,\sigma)$ can be studied by interpolation of the spaces $L^2(\bord,\sigma)$ and $H^1(\bord,\sigma)$, where $\sigma$ denotes the surface measure on $\bord$~\cite[Lemma 3.4]{sylvester_inverse_1988}.
For an admissible domain however, the trace space can no longer be identified with $H^{\frac12}(\bord)$.
For that matter, we use the trace representative from~\eqref{Eq:Tr-Representative}.

\begin{lemma}\label{Lem:InequalityB}
Let $\Omega$ be an admissible domain of $\R^n$. Let $f\in\Hp$ and $\varphi\in \mathrm{Lip}(\Omega)$. Then, $\varphi f\in \Hp$ and it holds:
\begin{equation*}
\|\varphi f\|_{\Hp}\le \sqrt2\|\varphi\|_{W^{1,\infty}(\Omega)}\|f\|_{\Hp}.
\end{equation*}
\end{lemma}

\begin{proof}
Consider $u\in V_1(\Omega)$ such that $\Tr  u=f$.
Since $\varphi u\in H^1(\Omega)$,~\eqref{Eq:Tr-Representative} allows to find a representative of $\Tr(\varphi u)$, given for q.e. $x\in\bord$ by
\begin{align}\label{Eq:TraceModuling}
\Tr  (\varphi u)(x)&=\lim_{r\to0}\,\frac{1}{|\Omega\cap B_r(x)|}\int_{\Omega\cap B_r(x)}\varphi(y)u(y)\,\dy \nonumber\\
&=\lim_{r\to0}\,\frac{1}{|\Omega\cap B_r(x)|}\int_{\Omega\cap B_r(x)}\big(\varphi(x)+ \underset{y\to x}{O}(|y-x|)\big)u(y)\,\dy \nonumber\\
&=\varphi(x)\Tr  u(x),
\end{align}
by Lipschitz-continuity of $\varphi$, hence $\Tr  (\varphi u)=\varphi f$. Therefore,
\begin{align*}
\|\varphi f\|^2_{\Hp} &\le\|\varphi u\|^2_{H^1(\Omega)}\\
&\le \|\varphi\|_{L^\infty(\Omega)}^2\|u\|^2_{L^2(\Omega)}+\|\varphi\|^2_{W^{1,\infty}(\Omega)}\|u\|_{H^1(\Omega)}^2\\
&\le 2\|\varphi\|_{W^{1,\infty}(\Omega)}^2\|u\|_{H^1(\Omega)}^2=2\|\varphi\|_{W^{1,\infty}(\Omega)}^2\|f\|_{\Hp}^2,
\end{align*}
by Theorem~\ref{Th:Trace}, Point~\ref{Pt:Trace-isometry}.
\end{proof}

A similar estimate holds for normal derivatives modulated by a Lipschitz function.

\begin{lemma}\label{Lem:InequalityB'}
Let $\Omega$ be an admissible domain of $\R^n$. Let $g\in\Hm$ and $\varphi\in \mathrm{Lip}(\Omega)$. Then $\varphi g\in \Hm$ and it holds:
\begin{equation*}
\|\varphi g\|_{\Hm}\le \sqrt 2\|\varphi\|_{W^{1,\infty}(\Omega)}\|g\|_{\Hm}.
\end{equation*}

\end{lemma}

\begin{proof}
Consider $u\in V_1(\Omega)$ such that $\ddn{u}=g$. For all $v\in V_1(\Omega)$, it holds $\Tr (\varphi v)=\varphi\Tr  v$ by~\eqref{Eq:TraceModuling}, hence
\begin{align*}
\left|\left\langle\varphi\ddn{u},\Tr  v\right\rangle\right| &\le \int_\Omega{|\varphi uv|\,\dx}+\int_\Omega{|\nabla u\cdot (v\nabla \varphi+\varphi\nabla v)|\,\dx}\\
&\le \|u\|_{H^1(\Omega)}(\|\varphi\|_{W^{1,\infty}(\Omega)}\|v\|_{L^2(\Omega)}+\|\varphi\|_{L^\infty(\Omega)}\|\nabla v\|_{L^2(\Omega)^n})\\
&\le \sqrt 2\|\varphi\|_{W^{1,\infty}(\Omega)}\|u\|_{H^1(\Omega)}\|v\|_{H^1(\Omega)}\\
&\le \sqrt 2\|\varphi\|_{W^{1,\infty}(\Omega)}\|g\|_{\Hm}\|\!\Tr  v\|_{\Hp},
\end{align*}
using the isometric properties of the trace and the normal derivative (Theorem~\ref{Th:Trace}, Point~\ref{Pt:Trace-isometry} and Proposition~\ref{Prop:ddn-isom}). Since $\Hp=\Tr (V_1(\Omega))$, the estimate follows.
\end{proof}

Combining the estimates above allows to state the stability of the direct problem on admssible domains.

\begin{theorem}[Stability of the direct problem]\label{Th:Calderon-Direct}
Let $\Omega$ be an admissible domain of $\R^n$. Let $\gamma_1,\gamma_2\in\Tcal(\Omega)$. Then,
\begin{multline}\label{Eq:EstimDirect}
\|\Lambda^{\gamma_1}-\Lambda^{\gamma_2}\|\le c(\widO)\|\gamma_1-\gamma_2\|_{L^\infty}\bigg(1+\frac{\|\gamma_2\|_{L^\infty}}{\inf \gamma_1}\bigg)\times\\
\bigg(1+\frac{\|\gamma_1\|_{L^\infty}}{\inf \gamma_1}+\frac{\|\gamma_2\|_{L^\infty}}{\inf \gamma_2}\bigg).
\end{multline}
In addition, if $\gamma_1,\gamma_2\in \mathrm{Lip}(\Omega)$, then
\begin{multline}\label{Eq:EstimInterpolation}
\|(\Lambda^{\gamma_1}-\gamma_1\Lambda^1)-(\Lambda^{\gamma_2}-\gamma_2\Lambda^1)\| \le c(\widO)\|\gamma_1-\gamma_2\|_{L^\infty}\bigg(1+\frac{\|\gamma_2\|_{L^\infty}}{\inf \gamma_1}\bigg)\\
\times\bigg(1+\frac{\|\nabla\gamma_1\|_{L^\infty}}{\inf \gamma_1}+\frac{\|\nabla\gamma_2\|_{L^\infty}}{\inf \gamma_2}\bigg)
\bigg(1+\sqrt{\frac{\|\gamma_1\|_{L^\infty}}{\inf \gamma_1}}+\sqrt{\frac{\|\gamma_2\|_{L^\infty}}{\inf \gamma_2}}\bigg).
\end{multline}
\end{theorem}

\begin{proof}
Let $f\in\Hp$. Denote by $u_1$ and $u_2$ the weak solutions to the Dirichlet problem~\eqref{Eq:ConductivityDir} associated to $\gamma_1$ and $\gamma_2$ respectively. Then, by Green's formula, it holds:
\begin{equation*}
\big\langle(\Lambda^{\gamma_1}-\Lambda^{\gamma_2})f,f\big\rangle_{\Bcal'\!,\,\Bcal}
=\int_\Omega\big[(\gamma_1-\gamma_2)|\nabla u_1|^2+\gamma_2\nabla(u_1+u_2)\cdot\nabla(u_1-u_2)\big]\,\dx.
\end{equation*}
Equations~\eqref{Eq:EstimGrad-u} and~\eqref{Eq:Estim-vDiff} and the fact that $u_1-u_2=v_1-v_2$ yield~\eqref{Eq:EstimDirect}.

Since for $i\in\{1,2\}$, $(\Lambda^{\gamma_i}-\gamma_i\Lambda^1)f=\gamma_i\ddn{v_i}$ and by Lemma~\ref{Lem:InequalityB'}, it holds
\begin{align*}
\|(\Lambda^{\gamma_1}-\gamma_1\Lambda^1)-(\Lambda^{\gamma_2}-\gamma_2\Lambda^1)\| &=\left\|\gamma_1\ddn{v_1}-\gamma_2\ddn{v_2}\right\|_{\Hm} \\
&\le \sqrt2\|\gamma_1-\gamma_2\|_{W^{1,\infty}(\Omega)}\left\|\ddn{v_1}\right\|_{\Hm}\\
& \qquad +\sqrt2\|\gamma_2\|_{W^{1,\infty}(\Omega)}\left\|\ddn{(v_1-v_2)}\right\|_{\Hm}.
\end{align*}
Since the problem $\Delta u=\psi\in L^2(\Omega)$ is well-posed on $H^1_0(\Omega)$ and by Green's formula, it holds
\begin{align*}
\left\|\ddn{v_1}\right\|_{\Hm} &\le \|\nabla^*\nabla v_1\|_{H^1(\Omega)}+\|\Delta v_1\|_{L^2(\Omega)}\\
&\le \|\nabla v_1\|_{L^2(\Omega)^n}+\|\Delta v_1\|_{L^2(\Omega)}\\
&\le c(\widO)\|\Delta v_1\|_{L^2(\Omega)},
\end{align*}
given that $\|\nabla^*\|=\|\nabla\|\le 1$, where $\nabla^*$ is the adjoint to $\nabla:H^1(\Omega)\to L^2(\Omega)^n$, and similarly,
\begin{equation*}
\left\|\ddn{(v_1-v_2)}\right\|_{\Hm}\le c(\widO) \|\Delta(v_1-v_2)\|_{L^2(\Omega)},
\end{equation*}
which allows to conclude by~\eqref{Eq:EstimLap-v} and~\eqref{Eq:EstimLapDiff}.
\end{proof}

\begin{remark}
In Corollary~\ref{Cor:EstimateDiff} and Theorem~\ref{Th:Calderon-Direct}, the conductivities $\gamma_1$ and $\gamma_2$ play symmetric parts.
For that matter, the estimates stated there can be refined by considering the minimum between the right-hand sides in Eqs.~\eqref{Eq:Estim-vDiff},~\eqref{Eq:EstimDirect} and~\eqref{Eq:EstimInterpolation} and their respective variants replacing $(\gamma_1,\gamma_2)$ with $(\gamma_2,\gamma_1)$.
\end{remark}

\subsection{Inverse problem}\label{Subsec:Inverse-Problem}

We turn to the inverse problem, which consists in reconstructing the conductivity using the Poincaré-Steklov operator, that is proving the injectivity of the operator
\begin{equation}\label{Eq:Mapping-Inverse-Problem}
\Lambda^\gamma\in\mathcal{L}(\Hp,\Hm)\longmapsto\gamma\in\Tcal(\Omega).
\end{equation}
Although we must at least assume $\gamma\in\Tcal(\Omega)$ for it is the class of conductivities for which the Poincaré-Steklov operator was defined on admissible domains (Definition~\ref{Def:P-S}), the results we prove here will always be conditional in the sense that further hypotheses must be made on $\gamma$ for the statements to hold (namely regularity, ellipticity and boundedness).
Without those restrictions, the inverse problem is known to be ill-posed, as several conductivities can yield the same operator~\cite[p. 156]{alessandrini_stable_1988}.

\subsubsection{Stability at the boundary}\label{Subsubsec:Inverse-Stability-Boundary}

In this section, we prove the stability at the boundary of the inverse conductivity problem.
To do so, we adapt the methods used in~\cite{alessandrini_singular_1990}, including the singular solutions introduced there.
A condition however to the existence of those solutions is that the conductivity $\gamma$ -- which is in $W^{1,\infty}$ in~\cite{alessandrini_singular_1990} -- be defined on all of $\R^n$, rather than on $\Omega$ alone.
It was proved in~\cite[Theorem 7]{hajlasz_sobolev_2008} that such an extension can be performed if and only if $\Omega$ is \textit{uniformly locally quasi-convex}, meaning there exist constants $\varepsilon,\delta>0$ such that for all $x,y\in\Omega$ with $|x-y|<\delta$, there exists a rectifiable curve $\gamma\subset\Omega$ joining $x$ and $y$ of length inferior to $\varepsilon|x-y|$.
To circumvent this restriction, we directly assume the conductivities are Lipschitz.

\begin{lemma}[Singular solutions]\label{Lem:SingularSol}
Let $\Omega$ be an arbitrary domain of $\R^n$. Let $\ell,L>0$ and $\gamma_1,\gamma_2\in\mathrm{Lip}(\Omega)$ with $\ell\le\gamma_{1,2}$ and $\|\gamma_{1,2}\|_{W^{1,\infty}(\Omega)}\le L$. Let $z\in \overline{\Omega}^c$. There exist $u_1,u_2\in H^1(\Omega)$ such that $\nabla\cdot(\gamma_1\nabla u_1)=\nabla\cdot(\gamma_2\nabla u_2)=0$ weakly on $\Omega$ and constants $C,r_0>0$ such that
\begin{align*}
\forall x\in \Omega,\quad &|\nabla u_{1,2}(x)|\le C|x-z|^{-n},\\
\forall x\in B_{r_0}(z)\cap\Omega,\quad &(\nabla u_1\cdot\nabla u_2) (x)\ge |x-z|^{-2n},
\end{align*}
where $C$ and $r_0$ depend only on $n$, $\ell$, $L$ and $\diamO$.
\end{lemma}

\begin{proof}
Since the conductivities are Lipschitz on $\Omega$, a Whitney type extension can be performed to define them on $\R^n$ (see~\cite[Section VI, Theorem 3]{stein_singular_1970}).
For the construction of the solutions, see~\cite[Lemma 3.1]{alessandrini_singular_1990} with $m=1$, which does not depend on the boundary regularity.
Note that the dependence on $\Omega$ of the constants can be refined as a dependence on its diameter, for one only needs to find a radius $R$ large enough in~\cite[Theorem 1.1]{alessandrini_singular_1990} so that a ball of radius $R$ contains $\Omega$.
\end{proof}

The solutions from Lemma~\ref{Lem:SingularSol} are referred to as `singular' for the inner product of their gradients blows up near a fixed \textit{singularity point}, denoted above by $z\in\overline{\Omega}^c$.
Considering singular solutions with a singularity approaching a boundary point in a way which is not `too tangential' (in the sense of~\eqref{eq:Corkscrew} below) allows to prove an \textit{a priori} estimate on the conductivities at the boundary based on the associated Poincaré-Steklov operators, that is the stability of the inverse conductivity problem.
In the case of a Lipschitz domain~\cite{alessandrini_singular_1990}, that condition described by~\eqref{eq:Corkscrew} is verified. The boundary can even be endowed with a vector field which is quasi-normal, in the sense that its inner product with the normal vector field is bounded below by a postive constant.
This allows to mimic a $C^\infty$ boundary and perform that non-tangential approach of the singularity point.
This method relies on the fact that Lipschitz boundaries are smooth almost everywhere (with respect to the $(n-1)$-dimensional surface measure), which gives `enough room' to smoothen the sudden changes in normal direction at the corners.

In the proof of the following theorem, we will make use of the fact that any one-sided extension domain $\Omega$ is an $n$-set~\cite{hajlasz_sobolev_2008}, and denote by $c_\Omega>0$ the constant involved in~\eqref{Eq:n-set}, which depends only on $\Omega$ (and $n$). 

\begin{theorem}[Stability at the boundary]\label{Th:Calderon-Inv-Stability-Boundary}
Let $\Omega$ be an admissible domain of $\R^n$ such that
\begin{equation}\label{eq:Corkscrew}
\exists\delta,\rho>0,\;\forall x_0\in\bord,\;\forall r<\rho,\;\exists z\in \overline{\Omega}^c,\quad \delta r < d(z,\bord) \le |z-x_0| < r. 
\end{equation}
Let $\ell, L>0$. Let $\gamma_1,\gamma_2\in \mathrm{Lip}(\Omega)$ be such that $\ell\le\gamma_{1,2}$ and $\|\gamma_{1,2}\|_{W^{1,\infty}(\Omega)}\le L$. Then, it holds
\begin{equation}
\|\gamma_1-\gamma_2\|_{L^\infty(\bord)}\le c\,\|\Lambda^{\gamma_1}-\Lambda^{\gamma_2}\|_{\mathcal{L}(\Hp,\Hm)},
\end{equation}
where $c>0$ depends on $\ell$, $L$, $n$, $\diamO$ and $c_\Omega$ from~\eqref{Eq:n-set}.
\end{theorem}

\begin{proof}
Let $x_0\in\bord$ be such that $\|\gamma_1-\gamma_2\|_{L^\infty(\bord)}=|(\gamma_1-\gamma_2)(x_0)|$. Up to swapping the indices, we may assume $(\gamma_1-\gamma_2)(x_0)>0$. Since $\gamma_1-\gamma_2$ is $L$-Lipschitz continuous, it holds
\begin{equation*}
\forall x\in\overline{\Omega},\quad \|\gamma_1-\gamma_2\|_{L^\infty(\bord)}\le (\gamma_1-\gamma_2)(x)+2L|x-x_0|.
\end{equation*}
For $k\in\N$ large enough (so that $2^{-k}<\rho$), let $z_k\in\overline{\Omega}^c$ satisfy~\eqref{eq:Corkscrew} for $x_0$ and $r=2^{-k}$, and denote $\sigma_k:=|z_k-x_0|$. In other words, it holds
\begin{equation}\label{Eq:Corkscrew-Sigma_k}
\delta 2^{-k}<d(z_k,\bord)\le |z_k-x_0|=\sigma_k< 2^{-k}.
\end{equation}
Let $u_1^{(k)},u_2^{(k)}\in H^1(\Omega)$ be singular solutions at $z_k$ (in the sense of Lemma~\ref{Lem:SingularSol}) with $\nabla\cdot(\gamma_i\nabla u_i^{(k)})=0$ on $\Omega$ such that $\int_{\Omega}{u_i^{(k)}\,\dx}=0$ for $i\in\{1,2\}$.
Let $r>0$ be small enough (so that Lemma~\ref{Lem:SingularSol} applies). It holds:
\begin{align*}
\int_{\Omega\cap B_r(z_k)}{|x-z_k|^{-2n}\,\dx}&\le \int_{\Omega\cap B_r(z_k)}{\nabla u_1^{(k)}\cdot\nabla u_2^{(k)}\,\dx}\\
&\le \int_{\Omega}{\nabla u_1^{(k)}\cdot\nabla u_2^{(k)}\,\dx}+\int_{\Omega\backslash B_r(z_k)}{|\nabla u_1^{(k)}\cdot\nabla u_2^{(k)}|\,\dx},
\end{align*}
hence, by Lemma~\ref{Prop:P-S},
\begin{align}\label{Eq:Direct-Prob-Calc}
\|\gamma_1-\gamma_2\|_{L^\infty(\bord)}&\int_{\Omega\cap B_r(z_k)}{|x-z_k|^{-2n}\,\dx} \nonumber \\
&\le \|\Lambda^{\gamma_1}-\Lambda^{\gamma_2}\|\|\Tr  u_1^{(k)}\|_{\Hp}\|\Tr  u_2^{(k)}\|_{\Hp} \nonumber\\
&\qquad +\int_{\Omega\backslash B_r(z_k)}{|(\gamma_1-\gamma_2)(x)||x-z_k|^{-2n}\,\dx} \nonumber\\
& \qquad +2L\int_{\Omega\cap B_r(z_k)}{|x-x_0||x-z_k|^{-2n}\,\dx}.
\end{align}
By Lemma~\ref{Lem:SingularSol}, it holds $$\|\Tr  u_{1,2}^{(k)}\|^2_{\Hp}\le \|u_{1,2}^{(k)}\|^2_{H^1(\Omega)}\le c(n,\ell,L,\diamO)\,\sigma_k^{-n},$$
Moreover, for $x\in\Omega$, 
$$|x-z_k|\ge d(z_k,\Omega)>\delta \sigma_k,$$
by~\eqref{Eq:Corkscrew-Sigma_k}, hence
\begin{align*}
\int_{\Omega\cap B_r(z_k)}{|x-z_k|^{-2n}|x-x_0|\,\dx}&\le \int_{\Omega\cap B_r(z_k)}{|x-z_k|^{-2n}(|x-z_k|+\sigma_k)\,\dx}\\
&\le \int_{\R^n\backslash B_{\delta \sigma_k}(z_k)}{|x-z_k|^{-2n}(|x-z_k|+\sigma_k)\,\dx}\\
&\le \frac{(1+\delta)n-1}{n(n-1)}\sigma_k^{-n+1},\\
\end{align*}
and for $k$ large enough, $2\sigma_k<r$, so that
\begin{align}\label{Eq:MinorationSigmaK}
\int_{\Omega\cap B_r(z_k)}{|x-z_k|^{-2n}\,\dx}&\ge \int_{\Omega\cap B_{2\sigma_k}(z_k)}{|x-z_k|^{-2n}\,\dx}\nonumber\\
&\ge (2\sigma_k)^{-2n}\,|\Omega\cap B_{2\sigma_k}(z_k)|.
\end{align}
For $k\in\N$ large enough, since $\Omega$ is connected, there exists $\zeta_{k+1}\in\Omega$ such that $|x_0-\zeta_{k+1}|\le \frac{\sigma_k}{2}$. Then $|\zeta_{k+1}-z_k|\le\frac{3}{2}\sigma_k$, hence by~\eqref{Eq:MinorationSigmaK},
\begin{equation*}
\int_{\Omega\cap B_r(z_k)}{|x-z_k|^{-2n}\,\dx}
\ge (2\sigma_k)^{-2n}\,|B_{\frac{\sigma_k}{2}}(\zeta_k)\cap\Omega|\ge 2^{-3n}c_\Omega\sigma_k^{-n},
\end{equation*}
where $c_\Omega$ from~\eqref{Eq:n-set} depends only on $n$ and $\Omega$. Moreover,
\begin{equation*}
\int_{\Omega\backslash B_r(z_k)}{|x-z_k|^{-2n}\,\dx}\le |\Omega|r^{-2n}.
\end{equation*}
Therefore, multiplying both sides of~\eqref{Eq:Direct-Prob-Calc} by $\sigma_k^n$ yields: 
\begin{equation*}
\|\gamma_1-\gamma_2\|_{L^\infty(\bord)}\le 2^{3n}\frac{C}{c_\Omega}\|\Lambda^{\gamma_1}-\Lambda^{\gamma_2}\|_{\mathcal{L}(\Hp,\Hm)} + \underset{k\to\infty}{o}(1),
\end{equation*}
where $C>0$ comes from Lemma~\ref{Lem:SingularSol}, hence the estimate.
\end{proof}

Condition~\eqref{eq:Corkscrew} is known as the \textit{corkscrew condition} for $\overline{\Omega}^c$~\cite{azzam_new_2017,jerison_boundary_1982, nystrom_integrability_1996}.
Examples of domains satisfying that condition are Lipschitz domains and, more generally, NTA domains~\cite{nystrom_integrability_1996} such as the admissible domain from Figure~\ref{Fig:AdmissibleDomains}.

\subsubsection{Identification on the domain}\label{Subsubsec:Inverse-Identification-Domain}

Having proved the stability at the boundary for the inverse problem, we turn to the determination of the conductivity on the whole domain.
To deal with a such problem, it is usual to rely on the equivalence between the conductivity equation and Schrödinger's equation~\cite{alessandrini_determining_1991, caro_global_2016, sylvester_inverse_1988}, in the sense that $u\in H^1(\Omega)$ satisfies $\nabla\cdot(\gamma\nabla u)=0$ weakly on $\Omega$ is and only if $v:=\sqrt\gamma u$ is such that
\begin{equation}\label{Eq:Schrodinger}
-\Delta v+qv=0\qquad \mbox{weakly on }\Omega,
\end{equation}
where
\begin{equation}\label{Eq:q}
q:=\frac{\Delta(\sqrt\gamma)}{\sqrt\gamma}.
\end{equation}
The purpose is to use the existence of specific solutions to~\eqref{Eq:Schrodinger} when $n\ge3$, known as CGO (complex geometrical optics) or high frequency solutions ~\cite{sylvester_global_1987}, which are solutions of the form
\begin{equation}\label{Eq:CGO-form}
v_\xi(x)=e^{\xi\cdot x}(1+R(\xi,x)),\quad x\in\Omega,
\end{equation}
where $\xi\in\C^n$ is non-null and such that $\xi\cdot\xi=0$, and $R(\xi,\cdot)$ is `small' in some sense.

We update the results from~\cite{caro_global_2016} and prove the identification of the conductivity $\gamma$ on the domain based on the Poincaré-Steklov operator $\Lambda^\gamma$.
Formally, if $v_\xi$ is a CGO solution to~\eqref{Eq:Schrodinger}, it follows that $R(\xi,\cdot)$ is such that
\begin{equation*}
(-\Delta-2\xi\cdot\nabla+q)R(\xi,\cdot)=-q\quad\mbox{weakly on }\Omega,
\end{equation*}
so that the existence of a CGO solution can be linked to the differential operator $(-\Delta-2\xi\cdot\nabla+q)$.
An \textit{a priori} estimate for that operator on $\R^n$ is proved in~\cite[Proposition 2.4]{caro_global_2016}.
Beyond the properties of the conductivity, the estimate depends only on the diameter of the support of the solutions.
This allows to prove a counterpart of~\cite[Proposition 2.5]{caro_global_2016} on the existence of CGO solutions in the case of admissible domains, using the same arguments which rely mostly on the Riesz representation theorem.
In that case, and if $\Omega$ satisfies the exterior corkscrew condition~\eqref{eq:Corkscrew}, then two conductivities $\gamma_1,\gamma_2\in \mathrm{Lip}(\Omega)$ such that $0<\ell\le\gamma_{1,2}$ and $\Lambda^{\gamma_1}=\Lambda^{\gamma_2}$ can be extended to $\R^n$ by means of the same $\tilde\gamma\in \mathrm{Lip}(\overline{\Omega}^c)$ which is bounded below by $\ell$, equal to $1$ oustide of a certain ball and with a Lipschitz constant comparable to those of $\gamma_{1,2}$.
This follows from the fact that $\gamma_1|_{\bord}=\gamma_2|_{\bord}$ by Theorem~\ref{Th:Calderon-Inv-Stability-Boundary}, and a Whitney type extension as in~\cite[Section VI, Theorem 3]{stein_singular_1970}.
From there, taking the same steps as in~\cite[Section 3]{caro_global_2016} yields $\gamma_1=\gamma_2$ on the domain.

\begin{theorem}[Identification on the domain]\label{Th:Inverse-Identification-Domain}
Let $\Omega$ be an admissible domain of $\R^n$, $n\ge 3$. Assume $\Omega$ satisfies the exterior corkscrew condition~\eqref{eq:Corkscrew}. Let $\ell>0$, and $\gamma_1,\gamma_2\in \mathrm{Lip}(\Omega)$ be such that $\ell\le\gamma_{1,2}$. Then, if $\Lambda^{\gamma_1}=\Lambda^{\gamma_2}$, it holds $\gamma_1=\gamma_2$ on $\overline\Omega$.
\end{theorem}

\subsubsection{Stability on the domain}\label{Subsubsec:Inverse-Stability-Domain}

In this part, we prove to the stability of the inverse problem on an admissible domain, that is the continuity of the mapping $\Lambda^\gamma\mapsto\gamma$. To do so, we adapt the method from~\cite{alessandrini_stable_1988,alessandrini_determining_1991} which consist in proving a stability estimate for the equivalent Schrödinger equation~\eqref{Eq:Schrodinger}. 
The associated Poincaré-Steklov operator is defined by
\begin{equation*}
\tilde{\Lambda}^q:\Tr  v\in \Hp \longmapsto \ddn{v}\in \Hm,
\end{equation*}
where $q$ is defined by~\eqref{Eq:q} and $v$ is a weak solution to~\eqref{Eq:Schrodinger}.
Once again, the idea is to use CGO solutions $v\in H^1(\Omega)$, which will lead to assuming $n\ge3$.
We define the following space of conductivities, constant near the boundary:
\begin{multline*}
\Ucal(\Omega):=\big\{\eta\in W^{2,\infty}(\Omega)\;\big|\;\operatorname{ess\,inf}\eta>0\quad\mbox{and}\\
\exists\, V \mbox{ neighborhood of }\bord\mbox{ in }\Omega,\; \eta|_{V}\mbox{ is constant}\big\}.
\end{multline*}
If $\gamma\in\Ucal(\Omega)$, then $q\in L^\infty(\Omega)$. The method used in~\cite{alessandrini_stable_1988,alessandrini_determining_1991} involves higher order normal derivatives of the conductivities, which -- to our knowledge -- have yet to be defined in the case of admissible domains.
For that matter, we assume the conductivities are constant near the boundary, so that their weak normal derivatives are null.

\begin{lemma}\label{Lem:NormalConstant}
Let $\Omega$ be an admissible domain of $\R^n$. Let $u\in H^1(\Omega)$ with $\Delta u\in L^2(\Omega)$ be such that $u$ is constant on a neighbourhood of $\bord$ in $\Omega$. Then $\ddn{u}|_{\bord}=0$.
\end{lemma}

\begin{proof}
Consider a sequence of smooth domain $(\Omega_k)_{k\in\N}$ such that $\Omega=\bigcap_{k\in\N}\Omega_k$ (for instance, a $C^\infty$ approximation of the dyadic approximation of $\R^n\backslash\overline\Omega$ from~\cite[Theorem 2.1]{claret_convergence_2024}).
Then $u\in H^1(\Omega)$ constant near the boundary can be extended as $u_0\in H^1(\Omega_0$), constant on $\Omega_0\backslash\Omega$.
Let $v\in H^1(\Omega)$.
Since $\Omega$ is an extension domain, $v$ can be extended as $v_0\in H^1(\Omega_0)$.
Then, for $k\in\N$, $\frac{\partial u_0}{\partial n}|_{\bord_k}=0$ and Green's formula yields
\begin{equation*}
0=\int_{\Omega_k}{(\Delta u_0) v_0\,\dx}+\int_{\Omega_k}{\nabla u_0\cdot\nabla v_0\,\dx} \xrightarrow[k\to+\infty]{}\int_{\Omega}{(\Delta u) v\,\dx}+\int_{\Omega}{\nabla u\cdot\nabla v\,\dx},
\end{equation*}
by dominated convergence. This means that for all $v\in H^1(\Omega)$, it holds $\langle\ddn u|_{\bord},\Tr v\rangle_{\Hm,\Hp}=0$. Since the trace operator is surjective on $\Hp$, $\Tr (H^1(\Omega))=\Hp$, we can deduce $\ddn{u}|_{\bord}=0$.
\end{proof}

The idea is then to prove the stability estimate using a similar result for the equivalent Schrödinger's equation, using CGO solutions to~\eqref{Eq:Schrodinger} (of the form~\eqref{Eq:CGO-form}) as in~\cite{alessandrini_determining_1991,sylvester_global_1987}.

\begin{lemma}\label{Lem:CGOSolutions}
Let $\Omega$ be an admissible domain of $\R^n$, $n\ge3$. Let $\ell, L>0$. Let $\gamma\in \Ucal(\Omega)$ be such that $\ell\le\gamma$ and $\|\gamma\|_{W^{2,\infty}(\Omega)}\le L$. Define $q\in L^\infty(\Omega)$ by~\eqref{Eq:q}. For all $\xi\in\C^n$ non-null with $\xi\cdot\xi=0$, there exists a solution $v\in H^1(\Omega)$ to~\eqref{Eq:Schrodinger} in the form
\begin{equation*}
v(x)=e^{\xi\cdot x}(1+R(\xi,x)),\quad x\in\Omega,
\end{equation*}
where
\begin{equation}\label{Eq:CGO-R-Estimate}
\|R(\xi,\cdot)\|_{L^2(\Omega)}\le c(\Omega) \frac{\|q\|_{L^\infty(\Omega)}}{|\xi|}.
\end{equation}
The solution $v$ is said to be a complex geometrical optics (CGO) solution to~\eqref{Eq:Schrodinger}, and satisfies
\begin{equation*}
\|v\|_{H^1(\Omega)}\le c(\ell,L,\Omega)\, e^{c(\Omega)|\xi|}.
\end{equation*}
\end{lemma}

\begin{proof}
Near $\bord$, $\gamma$ is constant hence $q$ is null. Therefore, it can be extended to a smooth domain containing $\Omega$ without modifying its $L^\infty$ norm, and the result follows from the regular case~\cite[Theorem 1.1]{alessandrini_determining_1991}.
\end{proof}

As it was pointed out in~\cite{alessandrini_determining_1991}, when $|\xi|$ is `small' in the sense that $|\xi|\le\|q\|_{L^\infty(\Omega)}$, the solution $v=\sqrt\gamma$ is a CGO solution to~\eqref{Eq:Schrodinger} in the sense of Lemma~\ref{Lem:CGOSolutions}. In that case,~\eqref{Eq:CGO-R-Estimate} can be refined as
\begin{equation}\label{Eq:CGO-R-Estimate-refined}
\|R(\xi,\cdot)\|_{L^2(\Omega)}\le c(|\Omega|,\ell,L) \|q\|_{L^\infty(\Omega)}.
\end{equation}
The CGO solutions allow to state a stability estimate for Schrödinger's equation.

\begin{proposition}\label{Prop:Estim-Schrodinger}
Let $\Omega$ be an admissible domain of $\R^n$, $n\ge3$. Let $\ell,L>0$. Let $\gamma_1,\gamma_2\in\Ucal(\Omega)$ be such that $\ell\le\gamma_{1,2}$ and $\|\gamma_{1,2}\|_{W^{2,\infty}(\Omega)}\le L$. Define $q_1,q_2$ by~\eqref{Eq:q} for $\gamma_1$ and $\gamma_2$ respectively. Then, there exists a function $\tilde\omega:\R\to\R$ such that
\begin{equation}\label{Eq:StabilitySchrodinger}
\|q_1-q_2\|_{H^{-1}(\Omega)}\le \tilde\omega\big(\|\tilde{\Lambda}^{q_1}-\tilde{\Lambda}^{q_2}\|\big),
\end{equation}
and
\begin{equation*}
\forall t\in]0,1[,\quad\tilde\omega(t)\le c(n,\ell,L,\Omega)|\ln t\,|^{-\tilde\delta},
\end{equation*}
for some $\tilde\delta>0$ depending only on $n$.
\end{proposition}

\begin{proof}
Let $k\in\R^n$ and $r>0$. Let $\xi_1,\xi_2\in \C^n$ be such that, for $i\in\{1,2\}$,
\begin{equation*}
\xi_i\cdot\xi_i=0,\quad \xi_1+\xi_2=ik\quad\mbox{and}\quad |\xi_i|^2=\frac{|k|^2}2+2r^2.
\end{equation*}
Then, by Lemma~\ref{Lem:CGOSolutions}, there exist solutions $v_1$ and $v_2$ to~\eqref{Eq:Schrodinger} for $q_1$ and $q_2$ respectively in the form
\begin{equation*}
v_i(x)=e^{\xi_i\cdot x}(1+R(\xi_i,x)),\quad x\in\Omega,\;i\in\{1,2\}.
\end{equation*}
By Green's formula,
\begin{align*}
\langle(\tilde{\Lambda}^{q_1}-\tilde{\Lambda}^{q_2})v_2, v_1\rangle_{\Hm,\Hp}&=\int_{\Omega}{(q_1-q_2)v_1v_2\,\dx}\\
&=\int_\Omega(q_1-q_2)e^{ik\cdot x}\big[1+R(\xi_1,x)+R(\xi_2,x)\\
&\qquad+R(\xi_1,x)R(\xi_2,x)\big]\,\dx,
\end{align*}
so that, by~\eqref{Eq:CGO-R-Estimate-refined} when $|k|$ and $r$ are `small', and~\eqref{Eq:CGO-R-Estimate} otherwise,
\begin{equation*}
\left|((q_1-q_2)\mathds{1}_\Omega)^\wedge(k)-\langle(\tilde{\Lambda}^{q_1}-\tilde{\Lambda}^{q_2})v_2, v_1\rangle\right|\le \frac{c(\ell,L,\Omega)}{|k|+r}.
\end{equation*}
Hence,
\begin{equation*}
|((q_1-q_2)\mathds{1}_\Omega)^\wedge(k)|\le c(\ell,L,\Omega)\left(e^{c(\Omega)(|k|+r)}\|\tilde{\Lambda}^{q_1}-\tilde{\Lambda}^{q_2}\|+\frac 1{|k|+r}\right).
\end{equation*}
Therefore, for $\kappa>0$,
\begin{align*}
\|q_1-q_2\|_{H^{-1}(\Omega)}^2 &\le \|(q_1-q_2)\mathds{1}_\Omega\|_{H^{-1}(\R^n)}^2\\
&\le \int_{\R^n}\frac{|((q_1-q_2)\mathds{1}_\Omega)^\wedge(k)|^2}{1+|k|^2}\,\mathrm{d}k\\
&\le \int_{|k|<\kappa}{|((q_1-q_2)\mathds{1}_\Omega)^\wedge(k)|^2\,\mathrm{d}k}+\int_{|k|>\kappa}\frac{|((q_1-q_2)\mathds{1}_\Omega)^\wedge(k)|^2}{1+\kappa^2}\,\mathrm{d}k\\
&\le c(\ell, L,\Omega)\left(e^{c(\Omega)(\kappa+r)}\|\tilde\Lambda^{q_1}-\tilde\Lambda^{q_2}\|+\frac{\kappa^n}r+\frac 1{1+\kappa^2}\right).
\end{align*}
Minimizing in $(\kappa, r)$ yields the estimate, as in~\cite{alessandrini_stable_1988, alessandrini_determining_1991}.
\end{proof}

To exhibit the connection between the operators $\Lambda^\gamma$ and $\tilde{\Lambda}^q$ when $\gamma$ and $q$ are linked by~\eqref{Eq:q}, we prove the following lemma on the weak normal derivation of a product.

\begin{lemma}\label{Lem:NormalProduct}
Let $\Omega$ be an admissible domain of $\R^n$. Let $\phi,\psi\in H^1_\Delta(\Omega)$ be such that $\phi\psi\in H^1_\Delta(\Omega)$. Then it holds
\begin{equation*}
\forall \chi\in H^1(\Omega),\quad \left\langle\ddn{(\phi\psi)},\Tr \chi\right\rangle=\left\langle\ddn{\phi},\Tr (\psi\chi)\right\rangle+\left\langle\ddn{\psi},\Tr (\phi\chi)\right\rangle.
\end{equation*}
\end{lemma}

\begin{proof}
By Green's formula, it holds:
\begin{align*}
\left\langle\ddn{(\phi\psi)},\Tr \chi\right\rangle &= \int_{\Omega}{\Delta (\phi\psi) \chi\,\dx}+\int_{\Omega}{\nabla (\phi\psi)\cdot\nabla \chi\,\dx}\\
&= \int_{\Omega}{(\Delta \phi) \psi \chi\,\dx}+\int_{\Omega}{\nabla \phi\cdot\nabla (\psi\chi)\,\dx}\\
&\qquad +\int_{\Omega}{(\Delta \psi) \phi \chi\,\dx}+\int_{\Omega}{\nabla \psi\cdot\nabla (\phi\chi)\,\dx},
\end{align*}
and the identity follows.
\end{proof}

This identity allows to generalize the link between the Poincaré-Steklov operators for the conductivity problem and for Schrödinger's equation~\cite[p.~168]{alessandrini_stable_1988} to the case of admissible domains, and derive an estimate of the latter using the former.

\begin{lemma}
Let $\Omega$ be an admissible domain of $\R^n$. Let $\ell,L>0$. Let $\gamma\in\Ucal(\Omega)$ be such that $\ell\le\gamma$ and $\|\gamma\|_{W^{2,\infty}(\Omega)}\le L$. Let $q$ be defined by~\eqref{Eq:q}. Then, it holds
\begin{equation*}
\tilde{\Lambda}^q=\frac1{\sqrt\gamma}\Lambda^\gamma\left(\frac1{\sqrt{\gamma}}\,\cdot\right).
\end{equation*}
\end{lemma}

\begin{proof}
Let $u\in H^1(\Omega)$ be such that $\nabla\cdot(\gamma\nabla u)=0$ on $\Omega$ and $v=\sqrt\gamma u$ be the solution to the equivalent Schrödinger problem. By~\eqref{Eq:TraceModuling}, it holds $\Tr  v=\sqrt{\gamma}\Tr  u$. Applying Lemma~\ref{Lem:NormalProduct}, followed by~\eqref{Eq:TraceModuling} again and Lemma~\ref{Lem:NormalConstant} yields:
\begin{equation*}
\ddn{v}=\ddn{(\sqrt{\gamma}u)}=\sqrt\gamma\ddn{u}=\frac1{\sqrt\gamma}\gamma\ddn u,
\end{equation*}
hence the formula.
\end{proof}

\begin{proposition}\label{Prop:LinkP-SOperators}
Let $\Omega$ be an admissible domain of $\R^n$, $n\ge3$. Let $\ell,L>0$. For $\gamma_1,\gamma_2\in \Ucal(\Omega)$ such that $\ell\le \gamma_{1,2}$ and $\|\gamma_{1,2}\|_{W^{2,\infty}(\Omega)}\le L$, it holds
\begin{equation*}
\|\tilde{\Lambda}^{q_1}-\tilde{\Lambda}^{q_2}\|\le c(\ell,L) \big(\|\Lambda^{\gamma_1}-\Lambda^{\gamma_2}\|+\|\gamma_1-\gamma_2\|_{L^{\infty}(\bord)}\big).
\end{equation*}
\end{proposition}

\begin{proof}
We use a similar decomposition to~\cite[p. 168-169]{alessandrini_stable_1988}.
The estimate relies on the boundedness and ellipticity of the conductivities and the local Lipschitz properties of $x\mapsto x^{-\frac12}$. 
The fact that the conductivities are constant near the boundary allows to improve Lemmas~\ref{Lem:InequalityB} and~\ref{Lem:InequalityB'}: for $f\in\Hp$, $\|\gamma_if\|_{\Hp}=\|\gamma_i\|_{L^\infty(\bord)}\|f\|_{\Hp}$ for $i\in\{1,2\}$.
\end{proof}

Proposition~\ref{Prop:LinkP-SOperators} allows to restate the stability estimate from~\eqref{Eq:StabilitySchrodinger} with a right-hand side in terms of the Poincaré-Steklov operator for the conductivity problem.
All that is left is to express the left-hand side in terms of the conductivities to yield the estimate.

\begin{proposition}\label{Prop:Estim-Cond-Alpha}
Let $\Omega$ be an admissible domain of $\R^n$, $n\ge3$. Let $\ell,L>0$. Let $\gamma_1,\gamma_2\in \Ucal(\Omega)$ be such that $\ell\le \gamma_{1,2}$ and $\|\gamma_{1,2}\|_{W^{2,\infty}(\Omega)}\le L$. Define $q_1,q_2$ by~\eqref{Eq:q} for $\gamma_1$ and $\gamma_2$ respectively. Then, it holds
\begin{equation*}
\|\gamma_1-\gamma_2\|_{L^\infty(\Omega)}\le c(\|q_1-q_2\|_{H^{-1}(\Omega)}+\|\gamma_1-\gamma_2\|_{L^\infty(\bord)})^\alpha,
\end{equation*}
where $\alpha\in]0,1[$ depends only on $n$, and $c>0$ on $n$, $\ell$, $L$ and $\diamO$.
\end{proposition}

\begin{proof}
It is known (see for instance~\cite[p. 167]{alessandrini_stable_1988}) that
\begin{equation*}
\nabla\cdot\left(\sqrt{\gamma_1\gamma_2}\,\nabla\ln\frac{\gamma_1}{\gamma_2}\right)=2\sqrt{\gamma_1\gamma_2}(q_1-q_2) \quad\mbox{on }\Omega.
\end{equation*}
Therefore, the well-posedness of the conductivity problem (see Appendix~\ref{A-Subsec:Conductivity}) yields
\begin{equation}\label{Eq:EstimWP-log}
\left\|\ln\frac{\gamma_1}{\gamma_2}\right\|_{H^1(\Omega)}\le c(\ell,L,\widO)\left(\|q_1-q_2\|_{H^{-1}(\Omega)}+\left\|\ln\frac{\gamma_1}{\gamma_2}\right\|_{\Hp}\right).
\end{equation}
Since $\gamma_1$ and $\gamma_2$ are constant near $\bord$, it holds $\ln\frac{\gamma_1}{\gamma_2}=l\in\R$ on $\bord$, hence:
\begin{equation*}
\left\|\ln\frac{\gamma_1}{\gamma_2}\right\|_{\Hp}\le\|l\|_{H^1(\Omega)}\le \sqrt{|\Omega|}\,|l|=\sqrt{|\Omega|}\left\|\ln\frac{\gamma_1}{\gamma_2}\right\|_{L^\infty(\bord)}.
\end{equation*}
By the local Lipschitz properties of $x\mapsto\ln x$,~\eqref{Eq:EstimWP-log} yields
\begin{equation*}
\left\|\ln\frac{\gamma_1}{\gamma_2}\right\|_{H^1(\Omega)}\le c(\ell,L,|\Omega|,\widO)\left(\|q_1-q_2\|_{H^{-1}(\Omega)}+\left\|\gamma_1-\gamma_2\right\|_{L^\infty(\bord)}\right).
\end{equation*}
Finally, by the interpolation formula from~\cite[Theorem 7.3]{ladyzhenskaya_boundary_1985} with some $m>n$, $p=+\infty$ and $r=2$, it holds
\begin{equation*}
\left\|\ln\frac{\gamma_1}{\gamma_2}\right\|_{L^\infty(\Omega)}\le c(n,|\Omega|,m)\left\|\ln\frac{\gamma_1}{\gamma_2}\right\|_{W^{2,\infty}(\Omega)}^{1-\alpha}\left\|\ln\frac{\gamma_1}{\gamma_2}\right\|_{H^1(\Omega)}^\alpha,
\end{equation*}
for some $\alpha\in]0,1[$ depending only on $n$ (and $m$), which yields the estimate, by the local Lipschitz properties of $x\mapsto e^x$.
\end{proof}

Altogether, we deduce the stability estimate on the domain for Calder\'on's inverse problem on an admissible domain.

\begin{theorem}[Stability on the domain]\label{Th:Inverse-Stability-Domain}
Let $\Omega$ be an admissible domain of $\R^n$, $n\ge3$. Let $\ell,L>0$. Let $\gamma_1,\gamma_2\in\mathcal{U}(\Omega)$ be such that $\ell\le\gamma_{1,2}$ and $\|\gamma_{1,2}\|_{W^{2,\infty}(\Omega)}\le L$. Then, there exists a function $\omega:\R\to\R$ such that
\begin{equation*}
\|\gamma_1-\gamma_2\|_{L^\infty(\Omega)}\le \omega(\|\Lambda^{\gamma_1}-\Lambda^{\gamma_2}\|)
\end{equation*}
and
\begin{equation*}
\forall t\in]0,1[,\quad \omega(t)\le c(n,\ell,L,\Omega)|\ln t\,|^{-\delta},
\end{equation*}
for some $\delta>0$ depending only on $n$.
\end{theorem}

\begin{proof}
Starting from Proposition~\ref{Prop:Estim-Cond-Alpha}, the right hand-side can be controlled by means of the stability estimate for the Schrödinger problem from Proposition~\ref{Prop:Estim-Schrodinger}. Then, by the estimate from Proposition~\ref{Prop:LinkP-SOperators} and the boundary stability estimate for the conductivity problem from Theorem~\ref{Th:Calderon-Inv-Stability-Boundary}, it can be further dominated by a quantity which depends only on $\|\Lambda^{\gamma_1}-\Lambda^{\gamma_2}\|$. The estimate follows.
\end{proof}

\appendix

\section{Well-posedness and \textit{a priori} estimates}\label{A-Sec:WP-Estimates}

The purpose of this section is to prove \textit{a priori} estimates for the Laplace equation and the conductivity equation with Dirichlet boundary conditions in the case of admissible domains.
In particular, we are interested in the dependencies of the constants involved.
Throughout this section, $\Omega$ is an admissible domain of $\R^n$.

\subsection{Laplace equation}\label{A-Subsec:Laplace}

Consider the Laplace equation with Dirichlet boundary condition
\begin{equation}\label{Eq:Laplace-Dirichlet}
\begin{cases}
-\Delta u = 0,\\
\Tr u=f\in\Hp.
\end{cases}
\end{equation}
Letting $\varphi\in H^1(\Omega)$ such that $\Delta\varphi=\varphi$ and $\tr\varphi=f$, $u\in H^1(\Omega)$ is a weak solution to~\eqref{Eq:Laplace-Dirichlet} if and only if $w:=u-\varphi\in H^1_0(\Omega)$ is a weak solution to
\begin{equation*}
\begin{cases}
-\Delta w = \varphi,\\
\Tr w=0,
\end{cases}
\end{equation*}
understood in the sense of its variational formulation
\begin{equation}\label{Eq:Laplace-Dirichlet-Var}
\forall v\in H^1_0(\Omega),\quad \int_\Omega\nabla w\cdot\nabla v\,\dx=\int_{\Omega}\varphi v\,\dx.
\end{equation}
By Lax-Milgram's theorem,~\eqref{Eq:Laplace-Dirichlet-Var} is well-posed on $H^1_0(\Omega)$, and the unique solution $w$ satisfies
\begin{equation*}
\|w\|_{H^1_0(\Omega)}\le c(\widO)\|\varphi\|_{H^1(\Omega)}=c(\widO)\|f\|_{\Hp},
\end{equation*}
by Theorem~\ref{Th:Trace}, Point~\ref{Pt:Trace-isometry}, and Poincaré's inequality. Consequently, it holds
\begin{equation*}
\|u\|_{H^1(\Omega)}\le\|w\|_{H^1(\Omega)}+\|\varphi\|_{H^1(\Omega)}\le c(\widO)\|f\|_{\Hp},
\end{equation*}
by Poincaré's inequality once again.

\subsection{Conductivity equation}\label{A-Subsec:Conductivity}

Let $\ell,L>0$, and let $\gamma\in\Tcal(\Omega)$ be such that $\ell\le\gamma$ and $\|\gamma\|_{W^{1,\infty}(\Omega)}\le L$. Consider the non-homogeneous conductivity problem with Dirichlet boundary condition
\begin{equation}\label{Eq:Conductivity-Dirichlet}
\begin{cases}
\nabla\cdot(\gamma\nabla u)=h\in H^{-1}(\Omega),\\
\Tr u=f\in\Hp.
\end{cases}
\end{equation}
By linearity, a weak solution $u\in H^1(\Omega)$ can be decomposed as $u=u_f+u_h$, where $u_f\in H^1(\Omega)$ and $u_h\in H^1_0(\Omega)$ are weak solutions to~\eqref{Eq:Conductivity-Dirichlet} corresponding to $h=0$ and $f=0$ respectively.

Let us focus on $u_f$ first. Letting $\varphi\in H^1(\Omega)$ such that $\Delta\varphi=\varphi$ and $\tr\varphi=f$, $u\in H^1(\Omega)$ is a weak solution to~\eqref{Eq:Conductivity-Dirichlet} if and only if $w:=u-\varphi\in H^1_0(\Omega)$ is a weak solution to
\begin{equation*}
\begin{cases}
\nabla\cdot(\gamma\nabla w) = -\nabla\cdot(\gamma\nabla\varphi),\\
\Tr w=0,
\end{cases}
\end{equation*}
understood in the sense of its variational formulation
\begin{equation}\label{Eq:Conductivity-Dirichlet-Var}
\forall v\in H^1_0(\Omega),\quad \int_\Omega\gamma\nabla w\cdot\nabla v\,\dx=-\int_{\Omega}\gamma\nabla\varphi\cdot\nabla v\,\dx.
\end{equation}
By Lax-Milgram's theorem,~\eqref{Eq:Conductivity-Dirichlet-Var} is well-posed on $H^1_0(\Omega)$, and the unique solution $w$ satisfies
\begin{equation}\label{Eq:Conductivity-Lifted-Estimate}
\|w\|_{H^1_0(\Omega)}\le\frac{\inf\gamma}{\|\gamma\|_{L^\infty(\Omega)}}\|\varphi\|_{H^1(\Omega)}\le c(\ell,L)\|f\|_{\Hp},
\end{equation}
by Theorem~\ref{Th:Trace}, Point~\ref{Pt:Trace-isometry}. By Poincaré's inequality, it follows that
\begin{equation*}
\|u_f\|_{H^1(\Omega)}\le\|w\|_{H^1(\Omega)}+\|\varphi\|_{H^1(\Omega)}\le c(\ell,L,\widO)\|f\|_{\Hp}.
\end{equation*}

The problem solved by $u_h$ is understood in terms of its variational formulation
\begin{equation*}
\forall v\in H^1_0(\Omega),\quad \int_\Omega\gamma\nabla u_h\cdot \nabla v\,\dx=\langle h, v\rangle_{H^{-1}(\Omega),\,H^1_0(\Omega)},
\end{equation*}
which is well-posed on $H^1_0(\Omega)$. $u_h$ is uniquely defined and satisfies
\begin{equation*}
\|u_h\|_{H^1(\Omega)}\le c(\ell,\widO)\|u_h\|_{H^1_0(\Omega)}\le c(\ell,\widO)\|h\|_{H^{-1}(\Omega)}.
\end{equation*}

Altogether, the conductivity problem~\eqref{Eq:Conductivity-Dirichlet} is well-posed on $H^1(\Omega)$, and the unique weak solution $u$ satisfies
\begin{equation*}
\|u\|_{H^1(\Omega)}\le\|u_f\|_{H^1(\Omega)}+\|u_h\|_{H^1(\Omega)}\le c(\ell,L,\widO)(\|f\|_{\Hp}+\|h\|_{H^{-1}(\Omega)}).
\end{equation*}

\section{Proofs of Subsection~\ref{Subsec:Direct-Problem}}\label{A-Sec:Proofs}

This section gathers those among the proofs of Subsection~\ref{Subsec:Direct-Problem} which can be adapted from~\cite[Section 3]{sylvester_inverse_1988} in a rather straightforward manner.

\begin{proof}[Proof of Lemma~\ref{Lem:EstimateG}]
Problem~\eqref{Eq:CondProbG} is understood in terms of its variational formulation:
\begin{equation*}
\forall v\in H^1_0(\Omega),\quad\int_\Omega{\gamma\nabla u\cdot\nabla v\,\dx}=\int_\Omega{\eta\nabla \varphi\cdot\nabla v\,\dx},
\end{equation*}
which is well-posed on $H^1_0(\Omega)$ by Lax-Milgram's theorem. Hence (for $v=u$),
\begin{equation*}
(\inf \gamma)\|u\|_{H^1_0(\Omega)}^2\le \|\eta\|_{L^\infty}\int_\Omega{|\nabla \varphi\cdot\nabla u|\,\dx},
\end{equation*}
which yields~\eqref{Eq:EstimGrad}.

If $\gamma,\eta\in W^{1,\infty}(\Omega)$ and $\Delta \varphi\in L^2(\Omega)$, then
\begin{equation*}
\nabla\cdot(\gamma\nabla u)=\nabla\gamma\cdot\nabla u+\gamma\Delta u\quad\mbox{and}\quad \nabla\cdot(\eta\nabla \varphi)=\nabla\eta\cdot\nabla \varphi+\eta\Delta \varphi,
\end{equation*}
hence multiplying~\eqref{Eq:CondProbG} by $\Delta u\in L^2(\Omega)$ yields
\begin{equation*}
\int_\Omega{\gamma(\Delta u)^2\,\dx}=\int_\Omega{\Delta u\big(\eta\Delta \varphi+\nabla\eta\cdot\nabla \varphi-\nabla\gamma\cdot\nabla u\big)\,\dx},
\end{equation*}
which implies
\begin{equation*}
(\inf \gamma)\|\Delta u\|_{L^2(\Omega)}\le \|\eta\|_{L^\infty}\|\Delta \varphi\|_{L^2}+\|\nabla\eta\|_{L^\infty}\|\nabla \varphi\|_{L^2}+\|\nabla\gamma\|_{L^\infty}\|u\|_{H^1_0}.
\end{equation*}
Then,~\eqref{Eq:EstimGrad} yields~\eqref{Eq:EstimLap}.
\end{proof}

\begin{proof}[Proof of Proposition~\ref{Prop:EstimateV}]
It holds $\nabla\cdot(\gamma\nabla v)=\nabla\cdot(\gamma\nabla u_\Delta)$ weakly on $\Omega$ and $\Tr  v=0$. Proceeding as in the proof of Lemma~\ref{Lem:EstimateG}, only writing the variational formulation as
\begin{equation*}
\forall w\in H^1_0(\Omega),\quad\int_\Omega{(\sqrt{\gamma}\nabla v)\cdot(\sqrt{\gamma}\nabla w)\,\dx}=\int_\Omega{(\sqrt{\gamma}\nabla u_\Delta)\cdot(\sqrt{\gamma}\nabla w)\,\dx},
\end{equation*}
yields~\eqref{Eq:EstimGrad-v} and~\eqref{Eq:EstimGrad-u}, given the estimate $\|\nabla u_\Delta\|_{L^2(\Omega)^n}\le c(\widO)\|f\|_{\Hp}$ (see Subsection~\ref{A-Subsec:Laplace}). Then,~\eqref{Eq:EstimLap} yields~\eqref{Eq:EstimLap-v}.
\end{proof}

\begin{proof}[Proof of Corollary~\ref{Cor:EstimateDiff}]
It holds $\nabla\cdot(\gamma_1\nabla(v_1-v_2))=\nabla\cdot((\gamma_2-\gamma_1)\nabla (v_2-u_\Delta))$ weakly on $\Omega$ and $\Tr (v_1-v_2)=0$. Using Lemma~\ref{Lem:EstimateG} followed by Proposition~\ref{Prop:EstimateV} yields the result.
\end{proof}

\section*{Acknowledgments}
The authors are most grateful to Claude Bardos for comments which have contributed to motivate this work, as well as valuable insights on the topic and on the related litterature.
The authors also thank Giovanni Alessandrini for enlightening comments on his previous works on the subject, as well as Alexander Teplyaev for interesting conversations regarding aspects of this work.

\def\refname{References}
\bibliographystyle{siam}
\bibliography{BibGC.bib}

\end{document}